\documentclass[11pt,a4paper]{amsart}
\usepackage{enumitem}
\usepackage{amssymb}
\usepackage{hyperref}

\newcommand{\Aut}{\operatorname{Aut}}

\newcommand{\G}{\mathcal{G}}

\newcommand{\Graphs}{\mathcal{G}raphs}

\newcommand{\Cay}{\operatorname{Cay}}
\newcommand{\IRel}{{I\mathcal{R}el}}

\newcommand{\Sym}{\operatorname{Sym}}

\newcommand{\id}{\operatorname{id}}

\newtheorem{theorem}{Theorem}[section]
\newtheorem{lemma}[theorem]{Lemma}
\newtheorem{corollary}[theorem]{Corollary}

\theoremstyle{definition}

	\theoremstyle{definition}
	\newtheorem{rem&notation}[theorem]{Remark and notation}

\newtheorem{remark}[theorem]{Remark}

\theoremstyle{definition}
\newtheorem{definition}[theorem]{Definition}

\numberwithin{equation}{section}

\title[Representability of permutation representations on coalgebras]{Representability of permutation representations on coalgebras and the isomorphism problem}

\author{Cristina ~Costoya}
\address[C.~Costoya]{Departamento de Computaci\'on, \'Alxebra,
Universidade da Coru{\~n}a, Campus de Elvi{\~n}a, 15071  A Coru{\~n}a, Spain.}
\email[C.~Costoya]{cristina.costoya@udc.es}

\author{David ~M\'endez}
\address[D.~M\'endez]{
	Departamento de {\'A}lgebra, Geometr{\'\i}a y Topolog{\'\i}a,
	Universidad de M{\'a}\-la\-ga, Campus de Teatinos, 29071 M{\'a}laga,
	Spain.}
\email[A.~Viruel]{david.mendez@uma.es}

\author{Antonio ~Viruel}
\address[A.~Viruel]{
Departamento de {\'A}lgebra, Geometr{\'\i}a y Topolog{\'\i}a,
Universidad de M{\'a}\-la\-ga, Campus de Teatinos, 29071 M{\'a}laga,
Spain.}
\email[A.~Viruel]{viruel@uma.es}

\subjclass[2010]{Primary 20G05; Secondary 05E18, 16T15}

\thanks{
First and second authors are partially  supported by Ministerio de Econom{\'\i}a y Competitividad (Spain), grant MTM2016-79661-P (AEI/FEDER, UE, support included). Second author is partially supported by Ministerio de Educaci\'on, Cultura y Deporte (Spain) grant FPU14/05137. Second and third authors are partially supported by  Ministerio de Econom{\'\i}a y Competitividad (Spain),  grant MTM2016-78647-P (AEI/FEDER, UE, support included). }

\AtBeginDocument{%
   \def\MR#1{}
}

\begin{document}

\begin{abstract}
	Let $G$ be a group and let $\rho\colon G\to\Sym(V)$ be a permutation representation of $G$ on a set $V$. We prove that there is a faithful $G$-coalgebra $C$ such that $G$ arises as the image of the restriction of $\Aut(C)$ to $G(C)$, the set of grouplike elements of $C$. Furthermore, we show that $V$ can be regarded as a subset of $G(C)$ invariant through the $G$-action, and that the composition of the inclusion $G\hookrightarrow\Aut(C)$ with the restriction $\Aut(C)\to\Sym(V)$ is precisely $\rho$. We use these results to prove that isomorphism classes of certain families of groups can be distinguished through the coalgebras on which they act faithfully.
\end{abstract}

\maketitle

\section{Introduction}\label{section:introduction}

Given $X$ an object in a category $\mathcal{C}$, the study of its automorphism group, $\Aut(X)$, is a difficult task. In fact, even deciding which groups arise as the automorphism groups of objects in $\mathcal{C}$ is far from trivial, see \cite{Jon18}. Nonetheless, it is also rewarding, as it can be expected that distinguished objects have distinguished automorphism groups, which in turn may give valuable information regarding the object $X$. Not only that, but if we know the automorphism groups of enough objects in $\mathcal{C}$, we can also draw conclusions regarding the category itself. One clear example of this comes from representation theory, as the automorphism groups of the objects of a category tell us a lot about which groups may act on which objects. 

The category of rings is one instance in which the possible automorphism groups of objects have been extensively studied. Indeed, there are many references in the literature regarding the realisability of groups as automorphisms of rings (see for example \cite{CosMenVir19,CosVir14b,KapShe12} regarding the associative case, and \cite{GorPop03} on the non-associative one). However, very little is known about the problem of representing groups as automorphisms of coalgebras. Moreover, given that coalgebras are only truly dual of algebras in the finite-dimensional case, general results on automorphisms of coalgebras cannot be deduced from the preexisting literature on automorphism groups of rings.

This article aims at providing the first result on the realisability of groups as automorphisms of coalgebras. Not only are we successful on that task, we also prove results regarding the realisability of (not necessarily finite) permutation representations. The key ingredient is the classical solution to the group realisability question in graphs, \cite{Gro59,Fru39,Sab60}, which allows us to define the desired coalgebras associated to graphs and draw conclusions on the realisability of groups. Following this idea, in Definition \ref{definition:pathCoalgebra} we introduce a coalgebra $C(\G)$ associated to a given graph $\G$ whose automorphism group is related to that of the graph. Namely, we prove:

\begin{theorem}\label{theorem:coalgebraExactSeq}
	Let $\Bbbk$ be a field and let $\G$ be a digraph. There is a $\Bbbk$-coalgebra $C(\G)$ such that $G\big(C(\G)\big) = V(\G)$ and the restriction map $\Aut\big(C(\G)\big)\to \Sym\big(G(C(\G))\big) = \Sym\big(V(\G)\big)$ induces a split short exact sequence of groups
	\[1\longrightarrow\prod_{e\in E(\G)}\left(\Bbbk\rtimes\Bbbk^\times\right) \longrightarrow \Aut\big(C(\G)\big)\longrightarrow \Aut(\G)\longrightarrow 1.\]
\end{theorem}

In particular, since every group can be realised as the automorphism group of a digraph, \cite{Gro59,Sab60}, we obtain the following immediate consequence:

\begin{corollary}\label{corollary:groupRealisation}
	Let $\Bbbk$ be a field and let $G$ be a group. There is a $\Bbbk$-coalgebra $C$ such that $\Aut(C)\cong K\rtimes G$, where $K$ is a direct product of semidirect products of the form $\Bbbk\rtimes\Bbbk^\times$. Furthermore, $G$ is the image of the restriction of the automorphisms of $C$ to $\Sym\big(G(C)\big)$.
\end{corollary}

Consequently, every group arises as the permutation group induced by the restriction of the automorphisms of a coalgebra to its set of grouplike elements. Thus, it is natural to ask if every possible permutation group (or more generally, every permutation representation) arises in this way. This would follow from Theorem \ref{theorem:coalgebraExactSeq} if every permutation representation were realisable in the context of graphs. However, we know that to be false, \cite[Section 4]{Cam04}, \cite{Fru39}. 

Nonetheless, we do know that every finite permutation group appears as the restriction of the automorphism group of a graph to an invariant subset of its vertices, \cite[Theorem 1.1]{Bou69}. In Theorem \ref{theorem:graphAction} we generalise this result to include (non-necessarily finite) permutation representations. Using this result, we prove:

\begin{theorem}\label{theorem:coalgebraActionRealisation}
	Let $G$ be a group, $\Bbbk$ a field and $\rho\colon G\to \Sym(V)$ a permutation representation of $G$ on a set $V$. There exists a $G$-coalgebra $C$ such that:
	\begin{enumerate}
		\item $G$ acts faithfully on $C$, that is, the action induces a group monomorphism $G\hookrightarrow \Aut(C)$;
		\item the image of the restriction map $\Aut(C)\to\Sym\big(G(C)\big)$ is $G$;
		\item there is a subset $V\subset G(C)$ invariant through the $\Aut(C)$-action on $C$ and such that $\rho$ is the composition of the inclusion $G\hookrightarrow\Aut(C)$ with the restriction $\Aut(C)\to\Sym(V)$; and,
		\item there is a faithful action $\bar{\rho}\colon G\to \Sym\big(G(C)\setminus V\big)$ such that the composition of the inclusion $G\hookrightarrow\Aut(C)$ with the restriction map $\Aut(C)\to\Sym\big(G(C)\big)$ is $\rho\oplus\bar{\rho}$.
	\end{enumerate}
\end{theorem}

Finally, we want to study the isomorphism problem in the category of groups using the representation theory on coalgebras. Namely, we want to see how isomorphism classes of groups can be distinguished through the coalgebras on which they act. This kind of problem has been deeply studied in other contexts. For example, the problem of distinguishing groups through their linear representations received significant attention until Hertweek solved it in the negative in a celebrated paper, \cite{Her01}. In this paper, Hertweek proves that there are two non-isomorphic finite groups $G$ and $H$, both of order $2^{21} 97^{28}$, with the same integral group ring, which in particular implies that they both have equivalent linear representation theories.

One more recent example is \cite{CosVir14b}, where the authors deal with the isomorphism problem in groups using the representation theory on commutative differential graded algebras. They are able to show that groups in a family containing all finite groups can be distinguished through their faithful actions on these algebraic structures.

In this paper we prove two results regarding the isomorphism problem of groups through representations on coalgebras. The first result tells apart isomorphism classes of groups from a family wider than the one considered in \cite{CosVir14b}, but it requires that we focus on how the action looks like on grouplike elements. Recall that a group $G$ is co-Hopfian if it does not contain any proper subgroups isomorphic to itself, or equivalently, if any monomorphism $G\hookrightarrow G$ is an automorphism. For example, finite groups are clearly co-Hopfian. For this family of groups, we prove:

\begin{theorem}\label{theorem:isoProblemGrouplike}
	Let $\Bbbk$ be a field and let $G$ and $H$ be two co-Hopfian groups. The following statements are equivalent:
	\begin{enumerate}
		\item $G$ and $H$ are isomorphic; and,
		\item for any $\Bbbk$-coalgebra $C$, there is an action of $G$ on $C$ that restricts to a faithful action on $G(C)$ if and only if there is an action of $H$ on $C$ that restricts to a faithful action on $G(C)$.
	\end{enumerate}
\end{theorem}

For our second result regarding the isomorphism problem we do not focus on grouplike elements, but we need to further restrict the considered family of groups. Let $\Bbbk$ be a finite field of cardinality $p^n$, $p$ a prime. In Definition \ref{definition:groupClass} we introduce a family of groups $\mathfrak{G}_{p,n}$ for which we prove:

\begin{theorem}\label{theorem:isoProblemGeneral}
	Let $\Bbbk$ be a finite field of order $p^n$, $p$ prime. Let $G$ and $H$ be groups in $\mathfrak{G}_{p,n}$. The following are equivalent.
	\begin{enumerate}
		\item $G$ and $H$ are isomorphic; and,
		\item for every $\Bbbk$-coalgebra $C$, $G$ acts faithfully on $C$ if and only if $H$ acts faithfully on $C$.
	\end{enumerate}		
\end{theorem}

We remark that, although the family $\mathfrak{G}_{p,n}$ is smaller than the class of co-Hopfian groups, $\mathfrak{G}_{2,1}$ still contains all $2$-reduced groups, that is, all groups with no normal $2$-subgroups.

\emph{Outline of the paper}. In Section \ref{section:coalgebras} we introduce the coalgebra $C(\G)$, Definition \ref{definition:pathCoalgebra}, compute its automorphism group and prove Theorem \ref{theorem:coalgebraExactSeq}. Section \ref{section:graphs} deals with the realisation of permutation representations in the category of graphs, providing a generalisation of \cite[Theorem 1.1]{Bou69}. In order to do so, we first provide a solution using binary relational systems, Theorem \ref{theorem:mainIRel}, to then translate it to simple graphs, Theorem \ref{theorem:graphAction}. Finally, Section \ref{section:actions} is devoted to group actions on coalgebras. In this section, we use the results in the previous sections to first discuss the realisability of permutation representations on coalgebras, proving Theorem \ref{theorem:coalgebraActionRealisation}, and then consider the isomorphism problem, proving Theorem \ref{theorem:isoProblemGrouplike} and Theorem \ref{theorem:isoProblemGeneral}.

\section{From graphs to coalgebras}\label{section:coalgebras}

In this section we want to build, associated to a combinatorial object, a coalgebra on which a given group acts faithfully. Traditionally, coalgebras associated to combinatorial objects are defined based on quivers. However, as our constructions are mostly graph-theoretical, we work in the framework of directed graphs or digraphs.

Then, let $\G = \big(V(\G), E(\G)\big)$ be a digraph. We build, associated to $\G$, a coalgebra $C(\G)$ on which $\Aut(\G)$ acts faithfully. Furthermore, we show that the restriction of the $\Aut(\G)$-action to the set of grouplike elements of $C(\G)$ is also faithful. More precisely, the image of the obvious map $\Aut(C(\G))\to\Sym\big(G(C(\G))\big)$ is precisely $\Aut(\G)$. We do so in Theorem \ref{theorem:coalgebraExactSeq}, our main result for this section. Let us begin by introducing the coalgebra $C(\G)$. 

\begin{definition}\label{definition:pathCoalgebra}
	Let $\Bbbk$ be a field and let $\G$ be a digraph. We define a coalgebra $C(\G) = (C,\Delta,\varepsilon)$ where $C = \Bbbk\{v\mid v\in V(\G)\}\oplus \Bbbk\{e\mid e\in E(\G)\}$ and where
	\begin{itemize}
		\item for each $v\in V(\G)$, $\Delta(v) = v\otimes v$ and $\varepsilon(v) = 1$; and,
		\item for each $e = (v_1, v_2)\in E(\G)$, $\Delta(e) = v_1\otimes e + e\otimes v_2$ and $\varepsilon(e) = 0$.
	\end{itemize}
\end{definition}

\begin{remark}\label{remark:GrouplikeElements}
	The coalgebra $C(\G)$ corresponds to the degree 1 term of the coradical filtration of the path coalgebra of $\G$ regarded as a quiver, see \cite[Section 5.1]{Chi04}. In particular, the grouplike elements of $C(\G)$ are precisely the grouplike elements of the path coalgebra of $\G$, that is, the vertices of the graph. Therefore, $G\big(C(\G)\big) = V(\G)$.
\end{remark}

We now move on to the computation of the automorphism group of $C(\G)$. In order to do so, we first define a family of automorphisms of $C(\G)$, Lemma \ref{lemma:familyAut}, to then show that no other automorphisms of $C(\G)$ exist, Lemma \ref{lemma:computingAut}. By abuse of notation, given $\sigma\in\Aut(\G)$, we write $\sigma$ also to denote the self-map of $E(\G)$ that maps $e = (v_1,v_2) \in E(\G)$ to $\big(\sigma(v_1),\sigma(v_2)\big)\in E(\G)$, thus $\sigma(e) = \big(\sigma(v_1),\sigma(v_2)\big)$.

\begin{lemma}\label{lemma:familyAut}
	Let $\G$ be a digraph, $\Bbbk$ be a field and consider $C(\G)$ the coalgebra introduced in Definition \ref{definition:pathCoalgebra}. For any $\sigma\in\Aut(\G)$, and for any maps $\lambda\colon E(\G)\to\Bbbk$ and $\mu\colon E(\G)\to\Bbbk^\times$, the linear map $f_{\lambda,\mu}^\sigma\colon C(\G)\to C(\G)$ given by
	\[\begin{cases}
		f_{\lambda,\mu}^\sigma(v) = \sigma(v), & \text{for all $v\in V(\G)$},\\
		f_{\lambda,\mu}^\sigma(e) = \lambda(e)\big(\sigma(v_2)-\sigma(v_1)\big) + \mu(e)\sigma(e), & \text{for all $e = (v_1,v_2)\in E(\G)$}.
	\end{cases}\]
	is a coalgebra automorphism of $C(\G)$.
\end{lemma}

\begin{proof}
	First we have to prove that $f_{\lambda,\mu}^\sigma$ is a morphism of coalgebras. Thus we need to check that $\varepsilon\circ f_{\lambda,\mu}^\sigma = \varepsilon$ and that $\Delta\circ f_{\lambda,\mu}^\sigma = (f_{\lambda,\mu}^\sigma\otimes f_{\lambda,\mu}^\sigma)\circ\Delta$. We do the computations on the generators of $C(\G)$ associated to vertices and edges of $\G$ separately.

	Let $v\in V(\G)$. Regarding the counit:
		\begin{itemize}
			\item $\varepsilon(v) = 1$.
			\item $(\varepsilon\circ f_{\lambda,\mu}^\sigma)(v) = \varepsilon\big(\sigma(v)\big) = 1$.
		\end{itemize}
	Thus they are equal. Similarly, regarding the comultiplication:
		\begin{itemize}
			\item $(\Delta\circ f_{\lambda,\mu}^\sigma)(v) = \Delta\big(\sigma(v)\big) = \sigma(v)\otimes\sigma(v)$.
			\item $\big((f_{\lambda,\mu}^\sigma\otimes f_{\lambda,\mu}^\sigma)\circ\Delta\big)(v) = (f_{\lambda,\mu}^\sigma\otimes f_{\lambda,\mu}^\sigma)(v\otimes v) = f_{\lambda,\mu}^\sigma(v)\otimes f_{\lambda,\mu}^\sigma(v) = \sigma(v)\otimes\sigma(v)$.
		\end{itemize}
	Again they are equal.

	Now let us take $e = (v_1,v_2)\in E(\G)$. First, regarding the counit:
		\begin{itemize}
			\item $\varepsilon(e) = 0$.
			\item $(\varepsilon\circ f_{\lambda,\mu}^\sigma)(e) = \varepsilon\big(\lambda(e)(\sigma(v_2)-\sigma(v_1)) + \mu(e)\sigma(e)\big) = \lambda(e)\big(\varepsilon(\sigma(v_2)) - \varepsilon(\sigma(v_1))\big) + \mu(e)\varepsilon\big(\sigma(e)\big) = 0$.
		\end{itemize}
	Finally, regarding the comultiplication:
		\begin{itemize}
			\item $(\Delta\circ f_{\lambda,\mu}^\sigma)(e) = \Delta\big(\lambda(e)(\sigma(v_2)-\sigma(v_1)) + \mu(e)\sigma(e)\big)$\smallskip

			$\hspace{15pt} = \lambda(e)\big(\Delta(\sigma(v_2)) - \Delta(\sigma(v_1))\big) + \mu(e)\Delta\big(\sigma(e)\big)$\smallskip

			$\hspace{15pt} = \lambda(e)\big(\sigma(v_2)\otimes\sigma(v_2) - \sigma(v_1)\otimes\sigma(v_1)\big) + \mu(e)\big( \sigma(v_1)\otimes\sigma(e) + \sigma(e)\otimes\sigma(v_2) \big)$.\smallskip

			\item $\big((f_{\lambda,\mu}^\sigma\otimes f_{\lambda,\mu}^\sigma)\circ\Delta\big)(e) = (f_{\lambda,\mu}^\sigma\otimes f_{\lambda,\mu}^\sigma)(v_1\otimes e + e\otimes v_2) $\smallskip

			\hspace{15pt}$= f_{\lambda,\mu}^\sigma(v_1)\otimes f_{\lambda,\mu}^\sigma(e) + f_{\lambda,\mu}^\sigma(e) \otimes f_{\lambda,\mu}^\sigma(v_2)$\smallskip

			\hspace{13pt} $= \sigma(v_1)\otimes \big[\lambda(e)\big(\sigma(v_2)-\sigma(v_1)\big) + \mu(e)\big(\sigma(e)\big)\big] + \big[\lambda(e)\big(\sigma(v_2)-\sigma(v_1)\big) + \mu(e)\big(\sigma(e)\big)\big]\otimes\sigma(v_2)$\smallskip

			\hspace{15pt}$ = \lambda(e)\big(\sigma(v_2)\otimes\sigma(v_2) - \sigma(v_1)\otimes\sigma(v_1)\big) + \mu(e)\big( \sigma(v_1)\otimes\sigma(e) + \sigma(e)\otimes\sigma(v_2) \big)$.		
		\end{itemize}

	Consequently, $f_{\lambda,\mu}^\sigma$ is a morphism of coalgebras. It remains to prove that it is an automorphism. We do so by proving that $f_{-\frac{\lambda}{\mu}\sigma^{-1},\frac{1}{\mu}\sigma^{-1}}^{\sigma^{-1}}$ is its inverse. We first consider the composition $f_{-\frac{\lambda}{\mu}\sigma^{-1},\frac{1}{\mu}\sigma^{-1}}^{\sigma^{-1}}\circ f_{\lambda,\mu}^\sigma$:
	\begin{itemize}
		\item For $v\in V(\G)$,
			\[(f_{-\frac{\lambda}{\mu}\sigma^{-1},\frac{1}{\mu}\sigma^{-1}}^{\sigma^{-1}}\circ f_{\lambda,\mu}^\sigma)(v) = (f_{-\frac{\lambda}{\mu},\frac{1}{\mu}}^{\sigma^{-1}})\big(\sigma(v)\big) = \sigma^{-1}\big(\sigma(v)\big) = v.\]

		\item For $e = (v_1, v_2)\in E(\G)$,
			\begin{align*}
				(f_{-\frac{\lambda}{\mu}\sigma^{-1},\frac{1}{\mu}\sigma^{-1}}^{\sigma^{-1}}&\circ f_{\lambda,\mu}^\sigma)(e) = (f_{-\frac{\lambda}{\mu},\frac{1}{\mu}}^{\sigma^{-1}})\big(\lambda(e)(\sigma(v_2)-\sigma(v_1)) + \mu(e)\sigma(e)\big) \\
				= \lambda&(e)\big(f_{-\frac{\lambda}{\mu}\sigma^{-1},\frac{1}{\mu}\sigma^{-1}}^{\sigma^{-1}}(\sigma(v_2))-f_{-\frac{\lambda}{\mu}\sigma^{-1},\frac{1}{\mu}\sigma^{-1}}^{\sigma^{-1}}(\sigma(v_1))\big) + \mu(e)f_{-\frac{\lambda}{\mu}\sigma^{-1},\frac{1}{\mu}\sigma^{-1}}^{\sigma^{-1}}\sigma(e) \\
				& = \lambda(e)(v_2 - v_1) + \mu(e)\left(-\frac{\lambda(e)}{\mu(e)}(v_2 - v_1) + \frac{1}{\mu(e)}(e)\right) = e.
			\end{align*}
	\end{itemize}
	We also have to consider the composition $f_{\lambda,\mu}^\sigma\circ f_{-\frac{\lambda}{\mu}\sigma^{-1},\frac{1}{\mu}\sigma^{-1}}^{\sigma^{-1}}$. However, notice that $f_{\lambda,\mu}^\sigma$ is recovered from $f_{-\frac{\lambda}{\mu}\sigma^{-1},\frac{1}{\mu}\sigma^{-1}}^{\sigma^{-1}}$ by performing on the indexes the same operations that we perform to $f_{\lambda,\mu}^\sigma$ to obtain $f_{-\frac{\lambda}{\mu}\sigma^{-1},\frac{1}{\mu}\sigma^{-1}}^{\sigma^{-1}}$. Consequently, the proof above already shows that $f_{\lambda,\mu}^\sigma\circ f_{-\frac{\lambda}{\mu}\sigma^{-1},\frac{1}{\mu}\sigma^{-1}}^{\sigma^{-1}}$ is the identity map. Then, $f_{\lambda,\mu}^{\sigma} \in \Aut\big(C(\G)\big)$.
\end{proof}

We now prove that every coalgebra automorphism of $C(\G)$ is of this form.

\begin{lemma}\label{lemma:computingAut}
	Let $\G$ be a digraph, $\Bbbk$ be a field and consider $C(\G)$ the coalgebra introduced in Definition \ref{definition:pathCoalgebra}. If $f\in\Aut\big(C(\G)\big)$, there exists $\sigma\in\Aut(\G)$ and two maps $\lambda\colon E(\G)\to\Bbbk$ and $\mu\colon E(\G)\to\Bbbk^\times$ such that $f$ is the coalgebra automorphism $f_{\lambda,\mu}^\sigma$ introduced in Lemma \ref{lemma:familyAut}.
\end{lemma}

\begin{proof}
	Let $f\in\Aut\big(C(\G)\big)$ be a coalgebra automorphism. First notice that any automorphism of coalgebras must permute the set of grouplike elements. By Remark \ref{remark:GrouplikeElements}, $G\big(C(\G)\big) = V(\G)$, thus there is a bijective map $\sigma\colon V(\G)\to V(\G)$ such that $f(v) = \sigma(v)$, for all $v\in V(\G)$.

	Now take $e\in E(\G)$. Then there are, for every $x\in V(\G)\cup E(\G)$, elements $\gamma(e,x)\in\Bbbk$ such that
	\begin{equation}\label{eq:f(x)}
		f(e) = \sum_{x\in V(\G)\cup E(\G)} \gamma(e,x) x.
	\end{equation}

	In order for $f$ to be a coalgebra morphism, it needs to verify that $\varepsilon\circ f = \varepsilon$ and that $(f\otimes f)\circ\Delta = \Delta\circ f$. We first consider the equality involving the counit. Recall from Definition \ref{definition:pathCoalgebra} that $\varepsilon(e) = 0$, for $e\in E(\G)$. Thus,
	\begin{equation}\label{eq:counitCompatibility}
		0 = \varepsilon\big(f(e)\big) = \varepsilon\left(\sum_{x\in V(\G)\cup E(\G)}\gamma(e,x)x\right) = \sum_{x\in V(\G)\cup E(\G)} \gamma(e,x)\varepsilon(x) = \sum_{v\in V(\G)} \gamma(e,v).
	\end{equation}

	Now consider the equality regarding the comultiplication. Take $e = (v_1, v_2)\in E(\G)$. Then,
	\begin{equation}\label{eq:deltafe}
		\begin{split}
			(\Delta\circ f)(e) = \Delta\left(\sum_{y\in V(\G)\cup E(\G)}\gamma(e,y)y\right) = \sum_{y\in V(\G)\cup E(\G)} \gamma(e,y) \Delta(y) \hspace{2.7cm}\\\hspace{2.7cm}
			= \sum_{w\in V(\G)} \gamma(e,w) w\otimes w + \sum_{h = (w_1,w_2)\in E(\G)} \gamma(e,h) [w_1\otimes h + h \otimes w_2].
		\end{split}
	\end{equation}
	On the other hand,
	\begin{equation}\label{eq:fdeltae}
		\begin{split}
			\big((f\otimes f)\circ\Delta\big)(e) = (f\otimes f) (v_1\otimes e + e\otimes v_2) = f(v_1)\otimes f(e) + f(e)\otimes f(v_2) \\
			= \sigma(v_1)\otimes\left(\sum_{y\in V(\G)\cup E(\G)} \gamma(e,y)y\right) + \left(\sum_{z\in V(\G)\cup E(\G)} \gamma(e,z)z\right)\otimes \sigma(v_2).
		\end{split}
	\end{equation}
	Equations \eqref{eq:deltafe} and \eqref{eq:fdeltae} must be equal. First,   notice that  $\sigma(v_1)\otimes\sigma(v_1)$ and $\sigma(v_2)\otimes\sigma(v_2)$ are the only summands of the form $w\otimes w$ with $w\in V(\G)$ that may arise in Equation \eqref{eq:fdeltae}. Thus, $\gamma(e,w) = 0$ if $w \ne \sigma(v_1),\sigma(v_2)$. Regarding the coefficients $\gamma\big(e,\sigma(v_1)\big)$ and $\gamma\big(e,\sigma(v_2)\big)$, notice that in Equation \eqref{eq:fdeltae} we have the summand
	\[\big[\gamma\big(e,\sigma(v_1)\big) + \gamma\big(e,\sigma(v_2)\big)\big]\sigma(v_1)\otimes\sigma(v_2),\]
	whereas $\sigma(v_1)\otimes\sigma(v_2)$ does not appear in Equation \eqref{eq:deltafe}. Consequently, $\gamma\big(e,\sigma(v_1)\big) = - \gamma\big(e,\sigma(v_2)\big)$. Moreover, and since no further restrictions exist regarding these coefficients, $\gamma\big(e,\sigma(v_2)\big)\in \Bbbk$.

	Finally, regarding the summands of the form $w_1\otimes (w_1, w_2) + (w_1, w_2)\otimes w_2$ arising in Equation \eqref{eq:deltafe}, the only possible non-trivial such summand in Equation $\eqref{eq:fdeltae}$ is $\sigma(v_1)\otimes \big(\sigma(v_1),\sigma(v_2)\big) + \big(\sigma(v_1),\sigma(v_2)\big) \otimes \sigma(v_2)$. Furthermore, the corresponding coefficient $\gamma\big(e,(\sigma(v_1),\sigma(v_2))\big)$ must be non-trivial, since otherwise $f$ would not be injective. We deduce that $\big(\sigma(v_1),\sigma(v_2)\big)\in E(\G)$, and as a consequence, $\sigma$ is in fact a morphism of graphs. An analogous reasoning for $f^{-1}\in\Aut\big(C(\G)\big)$ lets us deduce that $\sigma^{-1}$ is a morphism of graphs as well, so in fact $\sigma\in\Aut(\G)$. Regarding the coefficient, no further restrictions exist, so $\gamma\big(e,(\sigma(v_1),\sigma(v_2))\big)\in\Bbbk^\times$. 

	We have thus obtained that there is a graph automorphism $\sigma\in\Aut(\G)$ such that
	\[\begin{cases}
		f(v) = \sigma(v), & \text{for all $v\in V(\G)$}, \\
		f(e) = \gamma\big(e,\sigma(v_2)\big)\big(\sigma(v_2)-\sigma(v_1)\big) + \gamma\big(e,\sigma(e)\big)\sigma(e), & \text{for all $e=(v_1,v_2)\in E(\G)$},
	\end{cases}\]
	where $\gamma\big(e,\sigma(v_2)\big)\in\Bbbk$ and $\gamma\big(e,\sigma(e)\big)\in\Bbbk^\times$. Consequently, if for every $e = (v_1,v_2)\in E(\G)$ we define $\lambda(e) = \gamma\big(e,\sigma(v_2)\big)$ and $\mu(e) = \gamma\big(e,\sigma(e)\big)$, we obtain that $f = f_{\lambda,\mu}^\sigma$ as introduced in Lemma \ref{lemma:familyAut}. The result follows.
\end{proof}

Now that we have computed the automorphism group of the coalgebras $C(\G)$ introduced in Definition \ref{definition:pathCoalgebra}, we can prove the main result for this section, Theorem \ref{theorem:coalgebraExactSeq}.

\begin{proof}[Proof of Theorem \ref{theorem:coalgebraExactSeq}]
	Let $C(\G)$ be the coalgebra introduced in Definition \ref{definition:pathCoalgebra}. We shall prove that this is the desired coalgebra. As an immediate consequence of Lemma \ref{lemma:familyAut} and Lemma \ref{lemma:computingAut},
	\[\Aut\big(C(\G)\big) = \{f_{\lambda,\mu}^\sigma\mid \sigma\in\Aut(\G), \lambda\colon E(\G)\to\Bbbk, \mu\colon E(\G)\to\Bbbk^\times\}.\]
	In particular, since $G\big(C(\G)\big) = V(\G)$, the map $\Aut\big(C(\G)\big)\to \Sym\big(G(C(\G))\big) = \Sym\big(V(\G)\big)$ takes the automorphism $f_{\lambda,\mu}^\sigma\in \Aut\big(C(\G)\big)$ to $\sigma\in\Sym\big(V(\G)\big)$. Indeed, for all $v\in V(\G)$, $f^\sigma_{\lambda,\mu}(v) =\sigma(v)$. Therefore, the image of the map $\Aut\big(C(\G)\big)\to\Sym\big(V(\G)\big)$ is $\Aut(\G)$, whereas the kernel is
	\[K=\{f_{\lambda,\mu}^{\id_\G}\mid\lambda\colon E(\G)\to\Bbbk, \mu\colon E(\G)\to\Bbbk^\times \}.\]
	Let us define $f_{\lambda,\mu} = f_{\lambda,\mu}^{\id_\G}$. We now proceed to prove that $K\cong \prod_{e\in E(\G)}\left(\Bbbk\rtimes\Bbbk^\times\right)$.

	First, let us see how the group operation works in $K$. Thus take  $f_{\lambda_1,\mu_1}, f_{\lambda_2,\mu_2}\in K$. Then,
	\begin{itemize}
		\item For $v\in V(\G)$, $(f_{\lambda_2,\mu_2}\circ f_{\lambda_1,\mu_1})(v) = f_{\lambda_2,\mu_2}(v) = v$.
		\item For $e = (v_1,v_2)\in E(\G)$, 
		\begin{equation*}
			\begin{split}
				(f_{\lambda_2,\mu_2}\circ f_{\lambda_1,\mu_1})(e) = f_{\lambda_2,\mu_2}\big(\lambda_1(e)(v_2-v_1)  + \mu_1(e)e\big)\hspace{120pt}\\
				= \lambda_1(e)(v_2 - v_1) + \mu_1(e)\big(\lambda_2(e)(v_2 - v_1) + \mu_2(e)e\big)\hspace{50pt} \\
				\hspace{120pt} = \big(\lambda_1(e) + \mu_1(e)\lambda_2(e)\big)(v_2 - v_1) + \mu_1(e)\mu_2(e)e.
			\end{split}
		\end{equation*}
	\end{itemize}
	Consequently, $f_{\lambda_2,\mu_2}\circ f_{\lambda_1,\mu_1} = f_{\lambda_1 + \mu_1\lambda_2, \mu_1\mu_2}$. Thus, the group operation of $K$ acts independently on each of the elements of $E(\G)$. This implies that $K$ can be decomposed as a direct product of groups over $E(\G)$. Let us focus on one of the factors, thus pick an edge $e\in E(\G)$ and take
	\[K_e = \{f_{\lambda,\mu}\mid \lambda\colon E(\G)\to\Bbbk \text{ with $\lambda(e') = 0$ for all $e'\ne e$},\mu\colon E(\G)\to\Bbbk^\times \text{ with $\mu(e') = 1$ for $e'\ne e$}\}.\]
	Let us prove that $K_e$ is a semidirect product of the form $\Bbbk\rtimes\Bbbk^\times$.

	First, let us denote the maps taking every $e\in E(\G)$ to $0_\Bbbk$ and $1_\Bbbk$ by $0\colon E(\G)\to\Bbbk$ and $1\colon E(\G)\to\Bbbk^\times$ respectively. Now consider the subsets of $K_e$ given by $H_e = \{f_{\lambda,\mu}\in K_e\mid \lambda = 0\}$ and $N_e = \{f_{\lambda,\mu}\in K_e\mid \mu = 1\}$. Then, for $f_{0,\mu_1}, f_{0,\mu_2}\in H_e$, $f_{0,\mu_2}\circ f_{0,\mu_1} = f_{0, \mu_1\mu_2}$, so $H_e$ is a subgroup of $K$ isomorphic to $\Bbbk^\times$. Similarly, for $f_{\lambda_1,1},f_{\lambda_2,1}\in N_e$, $f_{\lambda_2,1}\circ f_{\lambda_1,1} = f_{\lambda_1+\lambda_2,1}$, thus $N_e$ is a subgroup of $K$ isomorphic to $\Bbbk$. Let us now check that $N_e\trianglelefteq K_e$ and that $K_e\cong N_e\rtimes H_e$. Consider the map
	\begin{align*}
		g_e\colon K_e & \longrightarrow H_e \\
		f_{\lambda,\mu}&\longmapsto f_{0,\mu}.
	\end{align*}
	Then simple computations show that $g_e$ is a group homomorphism. Moreover, it is clear that $N_e = \ker g_e$, which exhibits that $N_e\trianglelefteq K_e$ and that $K_e\cong N_e\rtimes H_e$. We deduce that
	\[K = \prod_{e\in E(\G)} K_e = \prod_{e\in E(\G)} (N_e\rtimes H_e) \cong \prod_{e\in E(\G)} (\Bbbk \rtimes \Bbbk^\times).\]

	Finally, let us see that the sequence is split. Define a map $\Aut(\G)\to\Aut\big(C(\G)\big)$ taking $\sigma\in\Aut(\G)$ to $f_{0,1}^\sigma$. Then, for $\sigma_1,\sigma_2\in\Aut(\G)$, a simple computation shows that $f_{0,1}^{\sigma_2}\circ f_{0,1}^{\sigma_1} = f_{0,1}^{\sigma_2\circ\sigma_1}$, thus it is a group homomorphism. Moreover, it is clearly a section of the restriction map $\Aut\big(C(\G)\big)\to\Aut(\G)$. The result follows.
\end{proof}

Then, since every group is the automorphism group of a graph, \cite{Gro59,Sab60}, Corollary \ref{corollary:groupRealisation} is an immediate consequence of Theorem \ref{theorem:coalgebraExactSeq}. Furthermore, since we know the sequence to be split, we can also compute the group of automorphisms of $C(\G)$ as a consequence of Theorem \ref{theorem:coalgebraExactSeq}, thus proving the following:

\begin{corollary}\label{corollary:autC(G)}
	Let $\Bbbk$ be a field and let $\G$ be a digraph. If $C(\G)$ is the coalgebra introduced in Definition \ref{definition:pathCoalgebra}, then
	\[\Aut\big(C(\G)\big) \cong \left(\prod_{e\in E(\G)}\left(\Bbbk\rtimes\Bbbk^\times\right)\right)\rtimes \Aut(\G).\]
\end{corollary}

\section{Graphs realising permutation groups}\label{section:graphs}

In Corollary \ref{corollary:groupRealisation} we have seen that every group can be realised as the permutation group induced by the restriction of the automorphisms of a coalgebra to its set of grouplike elements. We now ask if it is possible to realise every permutation group (or more generally, every permutation representation) in this context. As a consequence of the results contained in Section \ref{section:coalgebras}, this would hold if every permutation group were realisable in the context of graphs, that is, if for every permutation group $\rho\colon G\hookrightarrow\Sym(V)$ there was a graph $\G$ such that $V(\G)= V$ and $\Aut(\G)\cong G$. However, such a graph does not exist for all permutation groups, as shown in \cite[Section 4]{Cam04}, \cite{Fru39}.

In any case, if we allow the set of vertices to be enlarged, the next result can be proven.

\begin{theorem}[\cite{Bou69}, Theorem 1.1] \label{theorem:bouwer}
	Let $\rho \colon G\hookrightarrow\Sym(V)$ be a finite permutation group. There is a graph $\G$ such that
	\begin{enumerate}
	\item $V\subset V(\G)$, and $V$ is invariant through the automorphisms of $\G$;
	\item $G\cong \Aut(\G)$; and,
	\item the obvious restriction map $G\cong\Aut(\G)\rightarrow \Sym(V)$ is $\rho$.
	\end{enumerate}
\end{theorem}

In this section we generalise Theorem \ref{theorem:bouwer} to any permutation representation, see Theorem \ref{theorem:graphAction}. To do so, we first build objects solving the considered problem in the category of binary relational systems over a set $I$, or $\IRel$, and then translate it to graphs by a procedure called \emph{arrow replacement operation}, \cite[Section 4.4]{HelNes04}. We now introduce the category $\IRel$, following the notation of \cite{HelNes04}.

\begin{definition}
	Let $I$ be a set. A binary relational system over $I$, $\mathcal{S}$, is a set $V(\mathcal{S})$, called the set of vertices of $\mathcal{S}$, together with a family of binary relations $R_i(\mathcal{S})$ on $\mathcal{S}$, for $i\in I$, called edges of label $i$. Binary relational systems over a set $I$ also receive the name of binary $I$-systems. 

	A morphism of binary $I$-systems, $f\colon\mathcal{S}_1\to\mathcal{S}_2$, is a map $f\colon V(\mathcal{S}_1)\to V(\mathcal{S}_2)$ such that if $(v,w)\in R_i(\mathcal{S}_1)$ for some $i\in I$, then $\big(f(v), f(w)\big)\in R_i(\mathcal{S}_2)$. The category whose objects are binary relational systems over $I$ and whose morphisms are morphisms of binary relational systems over $I$ is denoted by $\IRel$.
\end{definition}

We now introduce the Cayley diagram, a classical construction of a binary $I$-system $\G$ whose group of automorphisms is $G$. It serves as the basic building block to our subsequent constructions.

\begin{definition}\label{definition:Cayley}
	Let $G$ be a group and let $S = \{s_j\mid j\in J\}$ be a generating set for $G$. The Cayley diagram of $G$ associated to $S$, $\Cay(G,S)$, is a binary $J$-system with $V\big(\Cay(G,S)\big) = G$ and where $(g,g')\in R_j\big(\Cay(G,S)\big)$ if and only if $g' = g s_j$.
\end{definition}

\begin{remark}\label{remark:fixvertex} 
	Recall from  \cite[Section 6]{Gro59} or \cite[Section 3.3]{CoxMos80} that $\Aut\big(\Cay(G,S)\big)\cong G$. An element $h\in G$ determines an automorphism of the Cayley diagram, which we denote $\phi_h: V \big(\Cay(G,S) \big) \to V \big(\Cay(G,S)\big) $, by taking the transitive group action on the set of vertices obtained by left multiplication by $h$; namely, for $g\in  V \big(\Cay(G,S) \big)$, $\phi_h(g) = hg$. Hence if $\phi_h$ fixes any vertex, it is the identity.
\end{remark}

We now proceed to define the binary $I$-system $\G$ giving a solution to our problem. 

\begin{definition}\label{definition:G}
	Let $\rho\colon G\to\Sym(V)$ be a permutation representation of $G$ on a set $V$ and let $S = \{s_j\mid j\in J\}$ be a generating set for $G$. Take $I = J \sqcup V$. Define $\G$ a binary $I$-system with vertex set $V(\G) = G\sqcup V$ and edges:
	\begin{itemize}
		\item for each $j\in J$ and for $g\in G$, $(g, gs_j) \in R_j(\G)$.
		\item for each $v\in V$ and for $g\in G$, $\big(g, \rho(g)(v)\big) \in R_v(\G)$.
	\end{itemize}
\end{definition}

\begin{remark}\label{remark:CayleyFullSubsystem}
	Notice that the full binary $I$-subsystem of $\G$ with vertex set $G$ is precisely $\Cay(G,S)$. We denote such subsystem by $\G(G)$. On the other hand, the full binary $I$-subsystem of $\G$ with vertex set $V$ has no edges.
\end{remark}

We now proceed to prove that $\Aut(\G)\cong G$. In order to do so, we first show that any element $\tilde{g}\in G$ induces an automorphism of $\G$:

\begin{lemma}\label{lemma:inducedMorphism}
	Consider the map
	\begin{align*}
		\Phi\colon G & \longrightarrow \Aut(\G)\\
		\tilde{g} & \longmapsto \Phi_{\tilde{g}}.
	\end{align*}
	where, for a given $\tilde{g}\in G$, the map $\Phi_{\tilde{g}}$ is defined as follows:
	\begin{itemize}
		\item For $g\in G$, define $\Phi_{\tilde{g}}(g) = \tilde{g} g$.
		\item For $v\in V$, define $\Phi_{\tilde{g}}(v) = \rho(\tilde{g})(v)$.
	\end{itemize}
	Then $\Phi$ is a group monomorphism.
\end{lemma}

\begin{proof}
	We first prove that $\Phi_{\tilde{g}}$ is a morphism of binary $I$-systems, that is, that it respects relations $R_i(\G)$, $i\in I$.
	\begin{itemize}
		\item For $g\in G$ and $j\in J$, $(g, g s_j)\in R_j(\G)$. And $\big(\Phi_{\tilde{g}}(g), \Phi_{\tilde{g}}(g s_j)\big) = (\tilde{g} g, \tilde{g} g s_j)\in R_j(\G)$.
		\item For $g\in G$ and $v\in V$, $\big(g, \rho(g)(v)\big)\in R_v(\G)$. And since $\rho$ is a group homomorphism, 
		\[\big(\Phi_{\tilde{g}}(g), \Phi_{\tilde{g}}(\rho(g)(v))\big) = \big(\tilde{g} g, \rho(\tilde{g})(\rho(g)(v))\big) = \big(\tilde{g} g, \rho(\tilde{g} g)(v)\big)\in R_v(\G).\]
	\end{itemize}

	Our next step is to prove that $\Phi$ is a group homomorphism, that is, that for $\tilde{g},\tilde{h}\in G$, $\Phi_{\tilde{g}}\circ \Phi_{\tilde{h}} = \Phi_{\tilde{g}\tilde{h}}$.
	\begin{itemize}
		\item For $g\in G$, 
			\[\big(\Phi_{\tilde{g}}\circ \Phi_{\tilde{h}}\big)(g) = \Phi_{\tilde{g}}\big(\Phi_{\tilde{h}} (g)\big) = \Phi_{\tilde{g}}(\tilde{h} g) = \tilde{g}\tilde{h} g = \Phi_{\tilde{g}\tilde{h}} (g).\]
		\item For $v\in V$ and since $\rho$ is a group homomorphism,
			\[\big(\Phi_{\tilde{g}}\circ \Phi_{\tilde{h}}\big)(v) = \Phi_{\tilde{g}}\big(\Phi_{\tilde{h}} (v)\big) = \Phi_{\tilde{g}}\big(\rho(\tilde{h})(v)\big) = \rho(\tilde{g})\big(\rho(\tilde{h})(v)\big) = \rho(\tilde{g}\tilde{h})(v) = \Phi_{\tilde{g}\tilde{h}} (v).\]
	\end{itemize}
	Thus $\Phi$ is a group homomorphism. Notice that, as a consequence, $\Phi_{\tilde{g}}$ is bijective, as $\Phi_{\tilde{g}} \circ \Phi_{\tilde{g}^{-1}} = \Phi_{\tilde{g}^{-1}} \circ \Phi_{\tilde{g}} = \Phi_{e_G}$ is clearly the identity. Finally, to show that $\Phi$ is a monomorphism notice that $\Phi_{\tilde{g}}(e_G) = \tilde{g}$, thus if $\tilde{g} \ne \tilde{h}$, $\Phi_{\tilde{g}} \ne \Phi_{\tilde{h}}$.
\end{proof}

To show that $\Aut(\G) \cong G$, it remains to prove that $\Phi$ is surjective:

\begin{lemma}\label{lemma:EveryAutInduced}
	For every $\psi\in \Aut(\G)$ there exists $\tilde{g}\in G$ such that $\psi = \Phi_{\tilde{g}}$.
\end{lemma}

\begin{proof}
	Take $\psi\in\Aut(\G)$. Notice that the only vertices of $\G$ that are starting vertices of edges labeled $v$ for some $v\in V$ are those in $G$. Thus, $\psi$ must leave $G$ invariant, so it must induce an automorphism on the full binary $I$-subsystem with vertex set $G$, that is, $\psi|_G \in \Aut\big(\G(G)\big)$. But recall from Remark \ref{remark:CayleyFullSubsystem} that $\G(G) \cong \Cay(G,S)$. Consequently, by Remark \ref{remark:fixvertex}, there exists $\tilde{g}\in G$ such that  $\psi|_G = \phi_{\tilde{g}}$. We shall prove that, in fact, $\psi = \Phi_{\tilde{g}}$.

	We already know that $\psi|_G = \Phi_{\tilde{g}}|_G$. It remains to prove the equality for vertices in $V$. Thus take $v\in V$. We know that $\big(e_G, \rho(e_G)(v)\big) = (e_G, v)\in R_v(\G)$. Then, $\big(\psi(e_G), \psi(v)\big) = \big(\phi_{\tilde{g}}(e_G), \psi(v)\big) = \big(\tilde{g}, \psi(v)\big)\in R_v(\G)$. But the only edge in $R_v(\G)$ starting at $\tilde{g}$ is $\big(\tilde{g}, \rho(\tilde{g})(v)\big)$. Thus $\psi(v) = \rho(\tilde{g})(v) = \Phi_{\tilde{g}}(v)$, for all $v\in V$. Then $\psi = \Phi_{\tilde{g}}$. The result follows.
\end{proof}

As a consequence of Lemma \ref{lemma:inducedMorphism} and Lemma \ref{lemma:EveryAutInduced} we immediately obtain the following:

\begin{corollary}
	$\Aut(\G)\cong G$. Moreover, every $\psi\in\Aut(\G)$ leaves $V\subset V(\G)$ invariant.
\end{corollary}

We finally need to consider what happens with the restriction of $\Aut(\G)$ to $V$.

\begin{lemma}\label{lemma:restriction}
	The restriction map $G\cong \Aut(\G)\to \Sym(V)$ is $\rho$. Moreover, there is a faithful action $\bar{\rho}\colon G\cong\Aut(\G)\to\Sym\big(V(\G)\setminus V\big)$ such that the restriction map $G\cong\Aut(\G) \to \Sym\big(V(\G)\big)$ is $\rho\oplus\bar{\rho}$.
\end{lemma}

\begin{proof}
	Let $g\in G$. Then $g$ is represented in $\Aut(\G)$ by $\Phi_g$, see Lemma \ref{lemma:inducedMorphism}. We first need to consider $\Phi_g|_V$. For each $v\in V$, by definition, $\Phi_g(v) = \rho(g)(v)$. Consequently, $\Phi_g|_V = \rho(g)$, for all $g\in G$.

	On the other hand, consider $\Phi_g|_{V(\G)\setminus V}$. Since $V(\G)\setminus V = G$, we have $e_G\in V(\G)\setminus V$. Moreover, $\Phi_g(e_G) = g$, for all $g\in G$. Consequently, the action $\bar{\rho}\colon G\to \Sym\big(V(\G)\setminus V\big)$ taking $g\in G$ to $\Phi_g|_{V(\G)\setminus V}$ is faithful. Moreover, the restriction map $G\cong\Aut(\G)\to\Sym\big(V(\G)\big)$ is $\rho\oplus\bar{\rho}$, as claimed.
\end{proof}

Finally, summing up Lemmas \ref{lemma:inducedMorphism}, \ref{lemma:EveryAutInduced} and \ref{lemma:restriction}, we deduce the following:

\begin{theorem}\label{theorem:mainIRel}
	Let $G$ be a group, $V$ be a set and $\rho\colon G\to\Sym(V)$ be a permutation representation of $G$ on $V$. There is a binary relational system $\G$ such that
	\begin{enumerate}
		\item $V\subset V(\G)$ and each $\psi\in\Aut(\G)$ is invariant on $V$;
		\item $\Aut(\G) \cong G$;
		\item the restriction $G\cong \Aut(\G)\to \Sym(V)$ is precisely $\rho$; and,
		\item there is a faithful action $\bar{\rho}\colon G\cong\Aut(\G)\to \Sym\big(V(\G)\setminus V\big)$ such that the restriction map $\G\cong\Aut(\G)\to\Sym\big(V(\G)\big)$ is $\rho\oplus\bar{\rho}$.
	\end{enumerate}
\end{theorem}

We now want to use results from \cite[Section 3]{CosMenVir19} to translate the construction in Theorem \ref{theorem:mainIRel} to simple graphs. The idea is to perform an \emph{arrow replacement operation}, following classical ideas by Frucht, \cite{Fru39}, and de Groot, \cite{Gro59}. The arrow replacement operation is a procedure by which the labeled edges on a binary relational system are replaced by a certain asymmetric graph (that is, a graph whose only automorphism is the identity map). These asymmetric graphs are chosen in such a way that automorphisms of the resulting simple graph must take the asymmetric graphs to copies of themselves, so they play the role of the labeled edges. The key idea to make this work is to ensure that the degrees of the vertices in the asymmetric graphs  are different to those of the vertices in the starting relational system. Consequently, we first need to compute possible degrees of vertices in our binary $I$-system $\G$, introduced in Definition \ref{definition:G}. Let us clarify what we mean by vertex degree.

\begin{definition}\label{def:degreeIRel}
	Let $\G$ be a binary $I$-system. For $v \in V(\G)$ we define:
	\begin{itemize}
		\item the \emph{indegree} of $v \in V(\G)$ as $\deg^-(v) = |\sqcup_{i\in I}\{w\in V(\G) \mid (w,v)\in R_i(\G)\}|$,
		\item the \emph{outdegree} of $v$ as $\deg^+(v) = |\sqcup_{i\in I}\{w\in V(\G) \mid (v, w)\in R_i(\G)\}|$,
		\item the \emph{degree} of $v$ as $\deg(v) = \deg^+(v) + \deg^-(v)$.
		\end{itemize}
	Observe that both the indegree and the outdegree of a vertex must be respected by the automorphisms of $\G$. In particular, the degree must also be respected.
\end{definition}

We now compute the degrees of vertices in the binary relational system from Definition \ref{definition:G}.

\begin{lemma}\label{lemma:degrees}
	Let $\G$ be the binary $I$-system introduced in Definition \ref{definition:G}.
	\begin{enumerate}
		\item Vertices in $G$ have degree $2|S| + |V|$.
		\item Vertices in $V$ have degree $|G|$.
	\end{enumerate}
\end{lemma}

\begin{proof}
	We begin by proving \emph{(1)}. Fix $g\in G$. First, recall from Remark \ref{remark:CayleyFullSubsystem} that the full binary subsystem with vertex set $G$ is $\G(G) = \Cay(G,S)$. Thus $g$ is the starting (respectively ending) vertex of exactly $|S|$ edges with labels in $S$. Furthermore, for each $v\in V$ there is precisely one edge labeled $v$ starting at $g$, and no more edges start or end in vertices in $G$. Thus, $\deg(g) = 2|S| + |V|$.

	Now take $v\in V$. Then, for each $g\in G$, $\rho(g)\in\Sym(V)$. This implies that each vertex $v\in V$ is connected with $g$ by exactly one edge. As this holds for every $g\in G$, and since there are no other edges starting or ending at $v\in V$, $\deg(v) = |G|$, for all $v\in V$. Thus \emph{(2)} follows.
\end{proof}

We can finally build a simple graph fulfilling the conditions stated at the beginning of this section. We remark that, although the coalgebra $C(\G)$ introduced in Definition \ref{definition:pathCoalgebra} is defined over a digraph, we construct simple graphs here since they provide a more general result, given that any graph can be regarded as a digraph where every edge is bidirected (that is, if $(v,w)$ is an edge in the digraph, $(w,v)$ is an edge as well). 

\begin{theorem}\label{theorem:graphAction}
	Let $G$ be a group, $V$ be a set and $\rho\colon G\to\Sym(V)$ be a permutation representation of $G$ on $V$. There is a (simple, undirected) graph $\G$ such that
	\begin{enumerate}
		\item $V\subset V(\G)$ and each $\psi\in\Aut(\G)$ is invariant on $V$;
		\item $\Aut(\G) \cong G$;
		\item the restriction $G\cong \Aut(\G)\to \Sym(V)$ is precisely $\rho$; and,
		\item there is a faithful action $\bar{\rho}\colon G\cong\Aut(\G)\to \Sym\big(V(\G)\setminus V\big)$ such that the restriction map $\G\cong\Aut(\G)\to\Sym\big(V(\G)\big)$ is $\rho\oplus\bar{\rho}$.
	\end{enumerate}
\end{theorem}

\begin{proof}
	Let $\G'$ be the binary $I$-system introduced in Definition \ref{definition:G}. As a consequence of Theorem \ref{theorem:mainIRel}, $\G'$ verifies properties analogous to \emph{(1)--(4)} in the category $\IRel$. We need to translate the solution from $\IRel$ to $\Graphs$.

	By Lemma \ref{lemma:degrees}, the set of possible degrees of vertices in $\G'$ is finite. Thus we can perform a replacement operation, following \cite[Section 3]{CosMenVir19}. In particular, in \cite[Subsection 3.1]{CosMenVir19}, the authors show that there is a (simple, undirected) graph $\G$ such that $V(\G')\subset V(\G)$ and $\Aut(\G) \cong \Aut(\G')$. Moreover, corresponding automorphisms in $\Aut(\G)$ and $\Aut(\G')$ induce the same map on $V(\G')$. In particular, $V\subset V(\G')\subset V(\G)$ is invariant through automorphisms of $\G$ and the restriction $G\cong \Aut(\G)\to \Sym(V)$ is equivalent to the restriction $G\cong \Aut(\G')\to\Sym(V)$ and, thus, it is precisely $\rho$. For the same reason, the automorphism of $\G$ associated to $g\in G$ takes $e_G\in V(\G')\setminus V\subset V(\G)\setminus V$ to $g\in V(\G')$, thus the action $\bar{\rho}\colon G\to V(\G)\setminus V$ is faithful and the restriction map $G\cong\Aut(\G)\to\Sym\big(V(\G)\big)$ is $\rho\oplus\bar{\rho}$. The result follows.
\end{proof}

\section{Permutation representations on coalgebras and the isomorphism problem}\label{section:actions}

In this section, we use the results proved so far to obtain conclusions regarding the representation theory of coalgebras. On the one hand, we show that the permutation representations of a group $G$ can be realised as the restriction of a $G$-action on a coalgebra to a certain subset of its grouplike elements, Theorem \ref{theorem:coalgebraActionRealisation}. On the other hand, we show that faithful coalgebra actions can be used in some cases to distinguish isomorphism classes of groups, Theorem \ref{theorem:isoProblemGrouplike} and Theorem \ref{theorem:isoProblemGeneral}.

Let us start with the realisability of permutation representations. As a consequence of Theorem \ref{theorem:coalgebraExactSeq}, the permutation group induced by the restriction of the automorphisms of $C(\G)$ to its set of grouplike elements is $\Aut(\G)$. We now translate Theorem \ref{theorem:graphAction} to coalgebras, proving Theorem \ref{theorem:coalgebraActionRealisation}.

\begin{proof}[Proof of Theorem \ref{theorem:coalgebraActionRealisation}]
	By Theorem \ref{theorem:graphAction}, there is a simple graph $\G$ such that $V\subset V(\G)$, $\Aut(\G)\cong G$, the restriction $G\cong \Aut(\G)\to \Sym(V)$ is $\rho$ and there is a faithful action $\bar{\rho}\colon G\cong\Aut(\G)\to\Sym\big(V(\G)\setminus V\big)$ such that the restriction map $\Aut(\G)\to\Sym(V)$ is $\rho\oplus\bar{\rho}$. Since any simple graph can be regarded as a digraph where every edge is bidirected, we can consider $C = C(\G)$ the coalgebra introduced in Definition \ref{definition:pathCoalgebra}. Then, $G(C) = V(\G)$. Let us prove that this is the desired coalgebra.

	Recall from Lemma \ref{lemma:familyAut} and Lemma \ref{lemma:computingAut} that the automorphisms of $C$ are the maps $f_{\lambda,\mu}^\sigma$ with $\sigma\in \Aut(\G)$, $\lambda\colon E(\G)\to\Bbbk$ and $\mu\colon E(\G)\to\Bbbk^\times$. Then since $G\cong\Aut(\G)$, $G$ acts on $C$ by taking an element $\sigma\in\Aut(\G)$ to $f_{0,1}^\sigma\in \Aut(C)$, thus $C$ is a $G$-coalgebra.

	On the other hand, for $v\in V(\G) = G(C)$, $f_{\lambda,\mu}^\sigma(v) = \sigma(v)$. Namely, the composition of the inclusion $G\cong\Aut(\G)\hookrightarrow\Aut(C)$ with the restriction $\Aut(C)\to\Sym\big(G(C)\big) = \Sym\big(V(\G)\big)$ is precisely the action of $G$ on $\G$ by automorphisms. The result then follows immediately from Theorem \ref{theorem:graphAction}.
\end{proof}

Finally, we review how we can use the results above to distinguish isomorphism classes of groups through their faithful representations on coalgebras and their restrictions to grouplike elements. Let us recall the concept of co-Hopfian group.

\begin{definition}
	A group $G$ is said to be co-Hopfian if it does not contain proper subgroups isomorphic to itself. Equivalently, every monomorphism $G\to G$ must be an automorphism.
\end{definition}

Clearly, every finite group is co-Hopfian. Other example of co-Hopfian groups are Artin groups, Tarski groups, \cite{Ols80, Ols82}, and fundamental groups of surfaces of genus at least two,  \cite[p. 58]{Har00}. We can now prove Theorem \ref{theorem:isoProblemGrouplike}.

\begin{proof}[Proof of Theorem \ref{theorem:isoProblemGrouplike}]
	One implication is obvious. Let us prove the remaining one. Suppose then that $G$ and $H$ are two groups verifying \emph{(2)}. Let us prove that $G\cong H$.

	Let $\G$ and $\mathcal{H}$ be digraphs such that $\Aut(\G) \cong G$ and $\Aut(\mathcal{H})\cong H$, which exist as a consequence of \cite[Section 6]{Gro59}, \cite{Sab60}. Consider the coalgebras $C(\G)$ and $C(\mathcal{H})$ introduced in Definition \ref{definition:pathCoalgebra}. As a consequence of Theorem \ref{theorem:coalgebraExactSeq}, $G$ acts faithfully on $C(\G)$, and the image of the composition of the inclusion map $G\to\Aut\big(C(\G)\big)$ with the restriction $\Aut\big(C(\G)\big)\to\Sym\big(G(C(\G))\big)$ is $G$. Therefore, there is an action of $G$ on $C(\G)$ that restricts to a faithful action on $G\big(C(\G)\big)$. By \emph{(2)}, this implies that there is an action of $H$ on $C(\G)$ that induces a faithful action on $G\big(C(\G)\big)$, so we deduce that $H\le \Aut(\G)\cong G$. Similarly, if there is an action of $G$ on $C(\mathcal{H})$ inducing a faithful action on $G\big(C(\mathcal{H})\big)$, then $G\le \Aut(\mathcal{H})\cong H$. Thus $G\le H\le G$ and, since $G$ is co-Hopfian, $G\cong H$.
\end{proof}

We now consider the entire action on the coalgebra instead of focusing on its restriction to grouplike elements. To ensure that groups are still distinguished, and since $\Aut\big(C(\G)\big)$ has subgroups of the form $\Bbbk\rtimes\Bbbk^\times$, we have to further restrict the class of groups we are working with. With such objective in mind, we introduce the following class of groups.

\begin{definition}\label{definition:groupClass}
	Let $\Bbbk$ be a finite field of order $p^n$, $p$ prime. A group $G$ is in the class $\mathfrak{G}_{p,n}$ if
	\begin{enumerate}
		\item $G$ is co-Hopfian; and,
		\item $G$ does not have finite non-trivial normal subgroups whose exponent divides $p^n(p^n-1)$.
	\end{enumerate}
\end{definition}

Notice that although this class is quite restrictive, it still contains many interesting groups. For example, $\mathfrak{G}_{2,1}$ still contains all $2$-reduced groups, that is, all groups with no normal $2$-subgroups. We can now prove Theorem \ref{theorem:isoProblemGeneral}.

\begin{proof}[Proof of Theorem \ref{theorem:isoProblemGeneral}]
	One implication is obvious. Let us prove the remaining one. Thus let $G$ and $H$ be two groups in $\mathfrak{G}_{p,n}$ verifying \emph{(2)} and let us prove that $G\cong H$.

	Again, let $\mathcal{G}$ and $\mathcal{H}$ be digraphs such that $\Aut(\G)\cong G$ and $\Aut(\mathcal{H})\cong H$, which exist by \cite[Section 6]{Gro59}, and consider $C(\G)$ and $C(\mathcal{H})$ the respective coalgebras from Definition \ref{definition:pathCoalgebra}. Then $G\cong \Aut(\G)$ acts faithfully on $C(\G)$ as an immediate consequence of Corollary \ref{corollary:autC(G)}. By the same result, if $H$ acts faithfully on $C(\G)$, there is a group monomorphism
	\[H\hookrightarrow \Aut\big(C(\G)\big)\cong \left(\prod_{e\in E(\G)}(\Bbbk\rtimes \Bbbk^\times)\right)\rtimes G.\]
	Thus $H$ is isomorphic to a subgroup of $\Aut\big(C(\G)\big)$, which we also denote by $H$. We shall see that $H\cap\big(\prod_{e\in E(\G)}(\Bbbk\rtimes\Bbbk^\times)\big)=\{1\}$.

	First notice that $\prod_{e\in E(\G)}(\Bbbk\rtimes\Bbbk^\times)$ is normal in $\Aut\big(C(\G)\big)$, thus $H\cap\big(\prod_{e\in E(\G)}(\Bbbk\rtimes\Bbbk^\times)\big)$ is normal in $H$. On the other hand, $\Bbbk\rtimes\Bbbk^\times$ is a group of order $p^n(p^n-1)$, thus the exponent of $\prod_{e\in E(\G)}(\Bbbk\rtimes\Bbbk^\times)$ divides $p^n(p^n-1)$. Therefore $H\cap\big(\prod_{e\in E(\G)}(\Bbbk\rtimes\Bbbk^\times)\big)$ is a normal subgroup of $H$ whose exponent divides $p^n(p^n-1)$. Hence, since $H$ is in $\mathfrak{G}_{p,n}$, the intersection must be the trivial group. Consequently, the image of $H$ falls in $G$, so $H\le G$.

	By a similar argument, we deduce that if $G$ acts faithfully on $C(\mathcal{H})$, then $G\le H$. We then have $G\le H\le G$ and, since $G$ is co-Hopfian, $G\cong H$.
\end{proof}


\providecommand{\bysame}{\leavevmode\hbox to3em{\hrulefill}\thinspace}
\providecommand{\MR}{\relax\ifhmode\unskip\space\fi MR }
\providecommand{\MRhref}[2]{%
  \href{http://www.ams.org/mathscinet-getitem?mr=#1}{#2}
}
\providecommand{\href}[2]{#2}

\end{document}